\documentclass[12pt,a4paper]{article}

\usepackage{latexsym,amssymb,amsfonts,amsmath,amsthm,nccmath,cases}
\usepackage{amsmath,amsthm,fullpage,mdwlist}
\usepackage{pdfsync}
\usepackage{amssymb, amsfonts, graphicx}
\usepackage[usenames,dvipsnames]{xcolor}
\usepackage{enumitem}
\usepackage[margin=2cm]{geometry}
\usepackage{multicol}

\newtheorem{theorem}{Theorem}

\newtheorem{definition}[theorem]{Definition}

\newtheorem{lemma}[theorem]{Lemma}

\newtheorem{problem}[theorem]{Problem}

\newcommand{\Cay}{\mathop{\mathrm{Cay}}}

\newcommand{\smod}{\hspace{-0.1cm} \pmod}
\newcommand{\ssmod}{\hspace{-0.2cm} \smod}

\let\oldproofname=\proofname
\renewcommand{\proofname}{\rm\bf{\oldproofname}}

\usepackage{authblk}

\begin{document}
\title{\textbf{On Hamilton Decompositions\\ of Infinite Circulant Graphs}}


\author[1]{Darryn Bryant \thanks{db@maths.uq.edu.au}}
\author[1]{Sarada Herke \thanks{s.herke@uq.edu.au}}
\author[1]{Barbara Maenhaut \thanks{bmm@maths.uq.edu.au}}
\author[2]{Bridget Webb\thanks{bridget.webb@open.ac.uk}}
\affil[1]{{\small School of Mathematics and Physics\\ The University of Queensland\\ QLD 4072, Australia} }
\affil[2]{{\small School of Mathematics and Statistics, The Open University, Milton Keynes, MK7 6AA, United Kingdom}}


\maketitle

\begin{abstract}

The natural infinite analogue of a (finite) Hamilton cycle is a two-way-infinite Hamilton path (connected spanning 2-valent subgraph). 
Although it is known that every connected $2k$-valent infinite circulant graph has a two-way-infinite Hamilton path, there exist many such graphs that do not have a decomposition into $k$ edge-disjoint two-way-infinite Hamilton paths. This contrasts with the finite case where it is conjectured that every $2k$-valent connected circulant graph
has a decomposition into $k$ edge-disjoint Hamilton cycles. We settle the problem of decomposing $2k$-valent infinite circulant graphs into $k$ edge-disjoint two-way-infinite Hamilton paths  
for $k=2$, in many cases when $k=3$, and in many other cases including where the connection set
is $\pm\{1,2,\ldots,k\}$ or $\pm\{1,2,\ldots,k-1,k+1\}$.

\end{abstract}

\section{Introduction and Notation\label{Into}}

Hamiltonicity refers to graph properties related to Hamilton cycles or Hamilton paths, and its study includes many classical problems in graph theory. 
One such problem is the Lov\'asz Conjecture \cite{Lo} which states that every finite connected vertex-transitive graph has a Hamilton path. There are only four non-trivial finite connected vertex-transitive graphs that are known to not admit a Hamilton cycle; the Petersen graph, the Coxeter graph, and the two graphs obtained from these by replacing each vertex with a triangle. None of these is a Cayley graph and a well-known conjecture states that 
every finite non-trivial connected Cayley graph has a Hamilton cycle, see \cite{WiGa}. 
Both the above-mentioned conjectures remain open.

A \textit{decomposition} of a graph is a set of edge-disjoint subgraphs which collectively contain all the edges; a decomposition into Hamilton cycles is called a \textit{Hamilton decomposition}, and a graph admitting a Hamilton decomposition is said to be {\textit{Hamilton-decomposable}. 
An obvious necessary condition for a Hamilton decomposition of a graph is that the graph be regular of even valency. Sometimes, a decomposition of a $(2k+1)$-valent graph into $k$ Hamilton cycles and a perfect matching is also called a Hamilton decomposition, but here we do not consider these to be Hamilton decompositions. 

In 1984, Alspach \cite{A} asked whether every $2k$-valent connected Cayley graph on a finite abelian group is Hamilton-decomposable.
It is known that every connected Cayley graph on a finite abelian group has a Hamilton cycle \cite{CheQui}, so it makes sense to consider the stronger property of Hamilton-decomposability.
Alspach's question is now commonly referred to as Alspach's conjecture. It holds trivially when $k=1$ and Bermond et al. proved that it holds for $k=2$ \cite{BeFaMa}. The case $k=3$ is still open, although many partial results exist, see \cite{Dea1,Dea2,W,Wes2,WesLiuKre}. There are also results for $k>3$, see \cite{AlsBryKre,FanLicLiu,Liu1,Liu2,Liu3}. It was shown in \cite{BryDea} that there exist $2k$-valent connected Cayley graphs on finite non-abelian groups that are not Hamilton-decomposable.

In this paper, we study the natural extension of Alspach's question to the case of Cayley graphs on infinite abelian groups, specifically in the case of the infinite cyclic group $\mathbb{Z}$.
We will not be considering any uncountably infinite graphs, so it should be assumed that the order of any graph in this paper is countable.
The natural infinite analogue of a (finite) Hamilton cycle is a \textit{two-way-infinite Hamilton path}, which is defined as a connected spanning $2$-valent subgraph.
This is, of course, an exact definition for a Hamilton cycle in the finite case, and accordingly we define a \textit{Hamilton decomposition} of an infinite graph 
to be a decomposition into two-way-infinite Hamilton paths. 
A \textit{one-way-infinite Hamilton path} is a connected spanning subgraph in which there is exactly one vertex of valency $1$, and the remaining vertices have valency $2$. 
For convenience, since we will not be dealing with one-way-infinite Hamilton paths, we refer to two-way-infinite Hamilton paths simply as \textit{Hamilton paths}, or as \textit{infinite Hamilton paths} if we wish to emphasise that the path is infinite.

Hamiltonicity of infinite circulant graphs, and of infinite graphs generally, has already been studied. In 1959, Nash-Williams \cite{NW} showed that every connected Cayley graph on a finitely-generated infinite abelian group has a Hamilton path. This result was proved again by Zhang and Huang in the special case of infinite circulant graphs \cite{ZH}.  There are also results on Hamilton paths in infinite Cayley digraphs, see \cite{DuJuWi, JunD, JunI} and Witte and Gallian's survey \cite{WiGa} on Hamilton cycles in Cayley graphs.

Hamiltonicity of infinite circulant graphs, and of infinite graphs generally, 
has already been studied. 
In 1959, Nash-Williams \cite{NW} showed that every connected Cayley graph on a finitely-generated infinite abelian group has a Hamilton path.
It seems that Nash-Williams' paper is largely unknown. For example, it is not cited in the 
1984 survey by Witte and Gallian \cite{WiGa}, and in 1995 Zhang and Huang \cite{ZH} proved the above-mentioned result of Nash-Williams in the special case of infinite circulant graphs.  
Indeed, D. Jungreis' paper \cite{JunD} on Hamilton paths in infinite Cayley digraphs is one of the few papers to cite Nash-Williams' result. Other results on Hamilton paths in infinite Cayley digraphs can be found in \cite{DuJuWi, JunI}.

Given the existence of Hamilton paths in Cayley graphs on 
finitely-generated infinite abelian groups, it makes sense to consider Hamilton-decomposability of these graphs. 
In this paper, we investigate this problem in the special case of infinite circulant graphs.
Witte \cite{Wi} proved that an infinite graph with infinite valency has a Hamilton decomposition if and only if it has infinite edge-connectivity and has a Hamilton path. By combining this characterisation with the result of Nash-Williams, we observe that if a connected Cayley graph on a finitely-generated infinite abelian group has infinite valency,
then it is Hamilton-decomposable, see Theorem \ref{Thm: inf S}. 

In Lemma \ref{NecCondition} we prove necessary conditions for an infinite circulant graph to be Hamilton-decomposable, thereby showing that not all connected infinite circulant graphs are Hamilton-decomposable. Since there are no elements of order $2$ in $\mathbb{Z}$, any infinite circulant graph with finite connection set is regular of valency $2k$ and is $2k$-edge-connected, for some non-negative integer $k$. Thus, neither the valency nor the edge-connectivity is an immediate obstacle to Hamilton-decomposability. 

We call infinite circulant graphs \textit{admissible} if they satisfy the necessary conditions for Hamilton-decomposability given in Lemma \ref{NecCondition} (see Definition \ref{d:admissible}). 
In Section \ref{sec:4Reg} we prove that all admissible 4-valent
infinite circulant graphs are Hamilton-decomposable.  We also show, in Section \ref{sec:2kReg}, that several other infinite families of infinite circulant graphs are Hamilton-decomposable, including many 6-valent infinite circulant graphs, and several families
with arbitrarily large finite valency.

Throughout the paper we make use of the following notation and terminology.
Let $\mathcal{G}$ be a group with identity $e$ and $S \subseteq \mathcal{G} - \{ e\}$ which is inverse-closed, that is, $s^{-1} \in S$ if and only if $s \in S$.  The \textit{Cayley graph} on the group $\mathcal{G}$ with \textit{connection set} $S$, denoted $ \Cay (\mathcal{G},S)$, is the undirected simple graph whose vertices are the elements of $\mathcal{G}$ and whose edge set is $\{ \{ g, gs \} \mid g \in \mathcal{G}, s \in S\}$.  When $\mathcal{G}$ is an infinite group, we call $\Cay(\mathcal{G},S)$ an \textit{infinite Cayley graph}.  
When $\mathcal{G}$ is a cyclic group, a Cayley graph $\Cay(\mathcal{G},S)$ is called a \textit{circulant graph}. 
Since we are interested in infinite circulant graphs, we will be considering graphs $\Cay(\mathbb{Z}, S)$, where $S$ is an inverse-closed set of distinct non-zero integers, which may be finite or infinite.  We define $S^+ = \{a \in S \mid a > 0\}$. Observe that if $|S^+|=k$ then $\Cay(\mathbb{Z}, S)$ is a $2k$-valent graph.

If $A$ is any subset of $\mathbb{Z}$ and $t\in\mathbb{Z}$, 
then we write $A+t$ to represent the set $\{ a+t \mid a \in A\}$.  
Furthermore, if $G$ is any graph with $V(G)\subseteq\mathbb{Z}$ and $t \in \mathbb{Z}$, 
then $G + t$ is the graph with vertex set
$\{x+t\mid x\in V(G)\}$ and 
edge set $\{ \{x+t, y+t\} \mid \{x,y\} \in E(G)\}$.
The \textit{length} of any edge $\{u,v\}$, denoted $\ell(u,v)$, in a graph with vertex set $\mathbb{Z}$ or $\mathbb{Z}_n$ is the distance from $u$ to $v$ in $\Cay(\mathbb{Z},\{\pm1\})$ or $\Cay(\mathbb{Z}_n,\{\pm 1\})$ if the vertex set is $\mathbb Z$ or $\mathbb Z_n$ respectively.

Next we discuss some notation for walks and paths in infinite circulant graphs which we will use throughout the remainder of the paper. The finite path with vertex set $\{v_1,v_2,\ldots,v_t\}$ and edge set 
$\{\{v_1,v_2\},\{v_2,v_3\},\ldots,\{v_{t-1},v_t\}\}$ is denoted $[v_1,v_2,\ldots,v_t]$. 
For $a \in \mathbb{Z}$ and $z_1, z_2, \dots, z_t \in S$, we define $\Omega_a( z_1, z_2, \dots, z_t)$ to be the walk in $\Cay(\mathbb{Z}, S)$ where the sequence of vertices is 
\[
a, \quad  a+z_1, \quad a+z_1+z_2, \quad  \dots \quad 
\quad a+ \sum_{i=1}^t z_i,
\]
so the lengths of the edges in the walk are $|z_1|, |z_2|, \dots, |z_t|$.  
Whenever we write $\Omega_a( z_1, z_2, \dots, z_t)$ it will be the case that $[a, a+z_1, \dots, a+\sum_{i=1}^t z_t]$ is a path 
, and we will use the notation $\Omega_a( z_1, z_2, \dots, z_t)$ interchangeably for both
the walk (with associated orientation, start and end vertices) and the path 
$[a, a+z_1, \dots, a+\sum_{i=1}^t z_t]$ (which is a graph with no inherent orientation).

\section{Necessary Conditions and Infinite Valency\label{sec:NecCond}} 


The following two lemmas give a characterisation of connected infinite circulant graphs, and 
necessary conditions for an infinite circulant graph to be Hamilton-decomposable.  We remark that the main idea in the proof of Lemma \ref{NecCondition} has been used in \cite{BrMa, HeMa}.

\begin{lemma} \textup{\cite{ZH}} \label{Lem: Connected}
If $S$ is an inverse-closed set of distinct non-zero integers and $\gcd(S) = d$, then $\Cay(\mathbb{Z}, S)$ has $d$ connected components which are each isomorphic to $\Cay(\mathbb{Z}, \{\frac{a}{d} \mid a \in S \})$.  In particular, $\Cay(\mathbb{Z}, S)$ is connected if and only if $\gcd(S) = 1$. 
\end{lemma}

\begin{lemma} \label{NecCondition}
If $\Cay(\mathbb{Z}, S)$ is Hamilton-decomposable, then 
\begin{itemize}
\item[(i)]  $S=\emptyset$ or $\gcd(S) = 1$; and
\item[(ii)] if $S$ is finite, then $\sum\limits_{a \in S^+}^{} a \equiv |S^+| \smod 2$.
\end{itemize}
\end{lemma}

\begin{proof}
Suppose $\Cay(\mathbb{Z}, S)$ has Hamilton decomposition $\mathcal{D}$. 
A Hamilton-decomposable graph is clearly either empty or connected, 
and so (i) follows immediately from Lemma \ref{Lem: Connected}.  
If $S$ is finite, then let $k=|S^+|$ and let $E=\{\{u,v\}\in E(\Cay(\mathbb{Z}, S)) \mid u\leq 0,v\geq 1\}$.
For each $a\in S^+$, there are exactly $a$ edges of length $a$ in $E$, and so we have 
$|E|=\sum\limits_{a \in S^+}^{} a$.
However, it is clear that each of the $k$ Hamilton paths in $\mathcal{D}$ 
has an odd number of edges from $E$. This means that $|E|\equiv k\smod 2$, and (ii) holds.
\end{proof}

It may seem plausible to use a Hamilton decomposition of a finite circulant graph 
to construct a Hamilton decomposition of an infinite circulant graph whose edges have the same lengths as those of the finite graph.
However, this is not possible in general.
For example, $\Cay(\mathbb{Z}_n, \pm\{1,2\})$ is Hamilton-decomposable for every $n \geq 5$, yet $\Cay(\mathbb{Z}, \pm\{1,2\})$ is not Hamilton-decomposable by Lemma \ref{NecCondition}.  


\begin{definition}\label{d:admissible} An infinite circulant graph $\Cay(\mathbb{Z}, S)$  is \textbf{admissible} if it satisfies (i) and (ii) from Lemma \ref{NecCondition}.  
\end{definition}

There are infinitely many connected infinite circulant graphs that are not admissible (and thus not Hamilton-decomposable).  For example, if $S$ is finite and $\Cay(\mathbb{Z}, S)$ is admissible, then for every even positive integer $s\notin S^+$, $\Cay(\mathbb{Z}, S \cup \pm\{s\})$ is not admissible.
We have found no admissible infinite circulant graphs that are not Hamilton-decomposable, and thus we pose the following problem. 

\begin{problem} \label{mainproblem}
Is every admissible infinite circulant graph Hamilton-decomposable?
\end{problem}

We now show that results from \cite{Wi} and \cite{NW} combine to settle this problem for the case where $S$ is infinite. 
An infinite graph $G$ with infinite valency is \textit{$\infty$-connected} if $G \setminus U$ is connected for every finite subset $U \subset V(G)$ (that is, $G$ has no finite cut-set).  An infinite graph $G$ with infinite valency has \textit{infinite edge-connectivity} if $G \setminus A$ is connected for every finite subset $A \subset E(G)$ (that is, $G$ has no finite edge-cut).  It is easy to see that if $G$ is $\infty$-connected then $G$ has infinite edge-connectivity, but the converse of this does not hold.


\begin{theorem} \textup{\cite{Wi}} \label{Lemma_inf_vert_trans}
Every vertex-transitive infinite graph $G$ of infinite valency that has a Hamilton path is $\infty$-connected.
\end{theorem}

\begin{theorem} \textup{\cite{Wi}} \label{Lemma_inf_ham_decomp}
A countably infinite graph of infinite valency has a Hamilton decomposition if and only if it has a Hamilton path and infinite edge-connectivity.
\end{theorem}


\begin{theorem} \textup{\cite{NW}} \label{Thm_Nash_Williams_ham_path}
Every connected Cayley graph on a finitely-generated infinite abelian group has a Hamilton path.
\end{theorem}

\begin{theorem}\label{Thm: inf S}
A Cayley graph of infinite valency on a finitely-generated infinite abelian group is Hamilton-decomposable if and only if it is connected.

\end{theorem}

\begin{proof}
Let $G$ be a Cayley graph of infinite valency on a finitely-generated infinite abelian group.
If $G$ is Hamilton-decomposable, then clearly it is connected. For the converse, suppose $G$ is connected.
Since $G$ is a Cayley graph it is vertex-transitive and so by Theorem 
\ref{Thm_Nash_Williams_ham_path} it has a Hamilton path. 
Thus, $G$ is $\infty$-connected by Theorem \ref{Lemma_inf_vert_trans}, and hence has 
infinite edge-connectivity.  
So $G$ is Hamilton-decomposable by Theorem \ref{Lemma_inf_ham_decomp}.
\end{proof}

Since an infinite circulant graph with infinite connection set is admissible if and only if it is connected,
Theorem \ref{Thm: inf S} answers Problem \ref{mainproblem} in the affirmative for the case of infinite connection sets.
For the remainder of this paper, we consider the case where the connection set is finite.  


\section{Infinite $4$-Valent Circulant Graphs \label{sec:4Reg}}

In this section we prove that all admissible $4$-valent infinite circulant graphs are Hamilton-decomposable, thus establishing the following theorem.

\begin{theorem} \label{Thm_ab}
A $4$-valent infinite circulant graph is Hamilton-decomposable if and only if it is admissible.
\end{theorem}

\begin{proof}  Let $G = \Cay(\mathbb{Z},\pm\{a,b\})$ be a $4$-valent admissible infinite circulant graph.  Thus $a$ and $b$ are distinct non-zero integers, and we may assume that $1 \leq a < b$.  Since $G$ is admissible, $\gcd(a,b) = 1$ and $a+b$ is even. Thus, $a$ and $b$ are both odd.
Let $t = b-a$, and define $m$ to be the integer in $\{0,1,\dots, t-1\}$ such that $m \equiv a \smod t$.  Note that $t$ is even and $\gcd(m, t) = 1$.  
For convenience, we use $b$ and $a+t$ interchangeably.

For each $i\in\{0,2,...,t-2\}$, define $\alpha_i$ to be the integer in $\{0,1,\ldots,t-1\}$ such that $\alpha_i\equiv im\smod t$, and define $\alpha_t=t$.
Let $F_v := \Omega_v(a,b)$ and $B_v := \Omega_v(a,-b)$.  Note that $F_v$ and $B_v$ correspond to paths $[v,v+a, v+2a+t]$ and $[v,v+a,v-t]$.
Now define $P$ as follows:
\[ P = \left( \bigcup_{i\in\{0,2,\ldots,t-2\}} F_{\alpha_i} \cup B_{2a+\alpha_i + t} \cup B_{2a+\alpha_i} \cup B_{2a+\alpha_i -t} \cup \dots \cup B_{t+\alpha_{(i+2)}} \right) \cup F_t, \]
see Figure \ref{fig_13_35_57}. It is straightforward to check that $P$ is a path with endpoints $0$ and $2b$ having $2b$ edges.
Since $\gcd(a,b) = 1$ and the lengths of the edges of $P$ alternate between $a$ and $b$, it follows that $P$ has exactly one vertex from each congruence class modulo $2b$ (except that the endpoints are both $0 \smod{2b}$).  
Hence, $H_1 =  \bigcup_{i \in \mathbb{Z}} \left( P+2bi \right)$ is a Hamilton path in $G$.  

The second Hamilton path is $H_2 =  \bigcup_{i \in \mathbb{Z}}\left( Q+2bi \right)$, where $Q = P+b$.  To see that $H_1$ and $H_2$ are edge-disjoint, it suffices to show that $E(Q)$ is disjoint from both $E(P)$ and $E(P+2b)$.  Observe that the edge set of $P$ is:
\begin{eqnarray}
E(P) &=& \{ \{x,x+a\} \mid x \textrm{ is even and } 0 \le x \le 2a+2t-2 \} \nonumber \nonumber \\ \nonumber 
	& & \cup \{ \{x,x+b\} \mid (x \textrm{ is even and } 2 \le x \le 2a+t-2) \textrm{ or } \\ \nonumber 
	& & \hspace{2.7cm}  (x \textrm{ is odd and } a \le x \le a+t)\}. \nonumber
\end{eqnarray} 

Since $E(Q) = \{ \{x+b, y+b\} \mid \{x,y\} \in E(P) \}$, where $b$ is an odd integer, it follows that
\begin{eqnarray}
E(Q) &=& \{ \{x,x+a\} \mid x \textrm{ is odd and } a+t \le x \le 3a+3t-2 \} \nonumber \nonumber \\ \nonumber 
	& & \cup \{ \{x,x+b\} \mid (x \textrm{ is odd and } a+t+2 \le x \le 3a+2t-2) \textrm{ or } \\ \nonumber
	& & \hspace{2.7cm} (x \textrm{ is even and } 2a+t \le x \le 2a+2t)\}. \nonumber
\end{eqnarray}
Thus, $E(Q) \cap E(P) = \emptyset$.  Similarly, $ E(Q)  \cap E(P+2b) = \emptyset$ and hence $H_1$ and $H_2$ form a Hamilton decomposition of $G$.  The result now follows by Lemma \ref{NecCondition}.

\begin{figure}[h!]
\begin{center}
\includegraphics[scale=1.4]{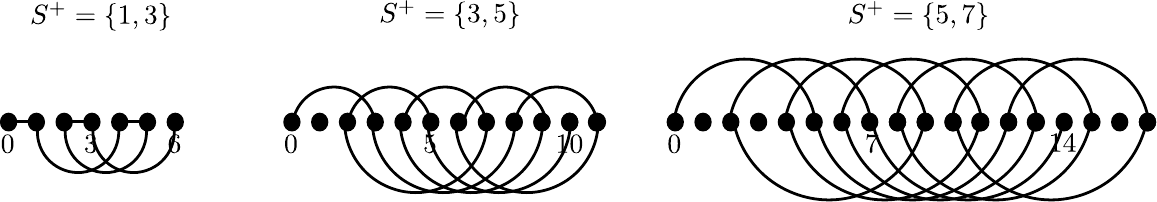}
\caption{Example constructions from Theorem \ref{Thm_ab} when $b-a = 2$.}
\label{fig_13_35_57}
\end{center}
\end{figure}
\end{proof}

\section{Infinite $2k$-Valent Circulant Graphs \label{sec:2kReg}}

In this section we consider $2k$-valent infinite circulant graphs for $k \geq 3$, proving that there are many infinite families of such graphs that are Hamilton decomposable if and only if they are admissible.

We begin by considering graphs $\Cay(\mathbb{Z}, S)$ where $|S^+| = k$ and $k \in S^+$ and no other element of $S^+$ is divisible by $k$.  An example of such a graph is $\Cay(\mathbb{Z}, \pm \{1,5,11,12,14\})$.  The next lemma, Lemma \ref{Lemma_suitable_path}, shows that in order to find a Hamilton decomposition of an admissible infinite circulant graph of this type, it suffices to find a Hamilton path in the complete graph with vertex set $\mathbb Z_k$ with edges of appropriate lengths. In the following lemma and its proof, the notation $[x]$ denotes the congruence class of $x$ modulo $k$.

\begin{lemma}\label{Lemma_suitable_path}
Suppose $k\geq 3$ is odd and $a_1,\dots, a_{k-1}$ are distinct positive integers which are not divisible by $k$ such that $G = \Cay(\mathbb{Z}, \pm\{a_1, \dots, a_{k-1},k\})$ is an admissible infinite circulant graph.  
If there exists a Hamilton path $Q$ in the complete graph with vertex set $\mathbb Z_k$ 
where the multiset 
$\{\ell(a,b)\mid\{a,b\}\in E(Q)\}$ equals the multiset
$\{\ell([0],[a_i])\mid i=1,2,\ldots,k-1\}$, then $G$ is Hamilton-decomposable.
\end{lemma}

\begin{proof} Let $Q = [v_0,v_1,\dots,v_{k-1}]$ be a Hamilton path in the complete graph with vertex set $\mathbb Z_k$ such that the multiset 
$\{\ell(a,b)\mid\{a,b\}\in E(Q)\}$ equals the multiset
$\{\ell([0],[a_i])\mid i=1,2,\ldots,k-1\}$. Thus, there exist integers 
$b_1,b_2,\ldots,b_{k-1}$ such that $[b_i]=[v_i]-[v_{i-1}]$ for $i=1,2,\ldots,k-1$ and 
$\pm\{b_1,b_2,\ldots,b_{k-1}\}=\pm\{a_1,a_2,\ldots,a_{k-1}\}$.
Define the path $P_1$ in $G$ by $P_1 = \Omega_0(b_1, b_2, \dots, b_{k-1})$.

Let $\Sigma = \sum_{i=1}^k b_i$ and note that $P_1$ is a path in $G$ with endpoints $0$ and $\Sigma$, and that $P_1$ has exactly one vertex from each congruence class modulo $k$.  Similarly, $P_1+k$ is a path in $G$ with endpoints $k$ and $k + \Sigma$, $P_1+k$ has exactly one vertex from each congruence class modulo $k$, and $P_1+k$ is disjoint from $P_1$.  

Now define $P$ to be the path in $G$ with edge set 
\[ E(P_1) \cup E(P_1 + k) \cup \{ \{\Sigma, \Sigma+k\}, \{k, 2k\} \}. \]

Thus, $P$ is a path with endpoints $0$ and $2k$ having $2k$ edges and having exactly one vertex from each congruence class modulo $2k$ (except that the the endpoints are both $0 \smod{2k}$).  Thus, \[H_1 = \bigcup_{i \in \mathbb{Z}} (P + 2ki)\] is a Hamilton path in $G$. 

Recall that $k$ is odd, so $\Sigma$ is even by Lemma \ref{NecCondition}.  The length $k$ edges in $P$ are $\{\Sigma, \Sigma + k\}$ and $\{k, 2k\}$, and the edges of length $a_i$ in $P$ are $\{x, x+a_i\}$ and $\{k+x, k+x+a_i\}$ for some integer $x$, for each $i = 1, \dots, k-1$.   
Thus, for each $d \in \{a_1,a_2,\dots,a_{k-1},k\}$, the set of length $d$ edges in $H_1$ is $\{ \{x,x+d\} \mid x \equiv x_1,x_2 \smod{2k}\}$ where $x_1$ is odd and $x_2$ is even and hence the length $d$ edges of each of $H_1,H_1+2,\dots, H_1+2k-2$ are mutually disjoint.
Thus $\{H_1, H_1+2, \dots, H_1+2k-2\}$ is a Hamilton decomposition of $G$.
\end{proof}

Buratti's Conjecture (see \cite{HR}) addresses the existence of paths with edges of specified lengths as required in Lemma \ref{Lemma_suitable_path}. It states that if $p$ is an odd prime and $L$ is a multiset containing $p-1$ elements from 
$\{1,2,\ldots,\frac{p-1}2\}$, then there exists a Hamilton path in the complete graph with vertex set $\mathbb Z_p$ such that the lengths of the edges of the path comprise the multiset $L$.
Buratti's Conjecture is open in general and, as noted by Horak and Rosa \cite{HR}, does not appear to be easy to solve.  By Lemma \ref{Lemma_suitable_path}, progress on Buratti's Conjecture can provide further constructions of Hamilton decompositions of admissible infinite circulant graphs.  Since the conjecture has been verified for $p \leq 23$ by Mariusz Meszka (see \cite{HR} for example), we have the following 
result.

\begin{theorem}\label{Cor_k_small_odd_prime}
If $p$ is an odd prime, where $p \leq 23$, and $a_1, a_2, \dots, a_{p-1}$ are distinct positive integers, not divisible by $p$, then $\Cay(\mathbb{Z}, \pm\{a_1, a_2, \dots, a_{p-1}, p\})$ is Hamilton-decomposable if and only if it is admissible.
\end{theorem}

Several cases and generalisations of Buratti's Conjecture have been studied in \cite{CF, DJ, HR}.  If $k$ is odd but not prime, it is not always possible to find a Hamilton path in $K_k$ with specified edge lengths.  For instance, there is no Hamilton path in $K_9$ where the multiset of edge lengths is one of the following: $\{1,3,3,3,3,3,3,3\}$, $\{2,3,3,3,3,3,3,3\}$, $\{3,3,3,3,3,3,3,3\}$ and $\{3,3,3,3,3,3,3,4\}$ \cite{DJ}.  


We note that when $k\geq 3$ is odd $Q= [0,1,k-1,2,k-2,3,k-3, \dots, \frac{k-1}{2}, \frac{k+1}{2}]$ is a Hamilton path in the complete graph with vertex set $\mathbb Z_k$ and the lengths of the edges of $Q$ comprise the multiset $\{1,1,2,2, \dots, \frac{k-1}{2}, \frac{k-1}{2} \}$. Thus, Lemma 10 immediately gives us the following result.


\begin{theorem}\label{WaleckiConstruction}
If $k\geq 3$ is odd and $a_1, \dots, a_{k-1}$ are distinct positive integers such that, for each $i=1,2,\dots,k-1$, $a_i \equiv i \smod{k}$, then $\Cay(\mathbb{Z}, \pm \{a_1, \dots, a_{k-1}, k\})$ is Hamilton-decomposable if and only if it is admissible.
\end{theorem}

We now address Hamilton-decomposability of $\Cay(\mathbb{Z}, \pm\{1,2,3,\dots,k\})$.

\begin{theorem} \label{ThmConsec}  If $k$ is any positive integer, then 
$\Cay(\mathbb{Z}, \pm\{1,2,3,\dots,k\})$ is Hamilton-decomposable if and only if it is admissible.
\end{theorem}

\begin{proof}  Let $G = \Cay(\mathbb{Z}, \pm\{1,2,\dots,k\})$.  Observe that $G$ is admissible if and only if $k \equiv 0,1 \smod{4}$.  If $k=1$, then the result is trivial, and if $k\geq 5$ with $k \equiv 1 \pmod{4}$, then $G$ is Hamilton-decomposable by Theorem \ref{WaleckiConstruction}. Thus, we can assume 
$k \equiv 0 \smod{4}$.  
If $k = 4$, then with $P = [0,-1,1,5,2,3,6,4,8]$, $H_1 = \cup_{i \in \mathbb{Z}} (P + 8i)$ is a Hamilton path in $G$ and $\{H_1, H_1+2, H_1+4, H_1 + 6 \}$ is a Hamilton decomposition of $G$.  

Now suppose $k \geq 8$ with $k \equiv 0 \smod{4}$ 
and let $u = \frac{k}{2}$ and $v = \frac{3k}{2}$. Define a path $P$ with $2k$ edges as follows.
\begin{align*} 
P := & \hspace{4mm} \Omega_0(-1, 2) \\
    & \cup  \Omega_1(k-2, -(k-3), k-4, -(k-5), \dots, 4,-3) \\
	& \cup \Omega_{u-1}(k, -(k-1), 1, k-1, -2) \\
	& \cup \Omega_{v-2}(3, -4, 5, -6, \dots, k-3, -(k-2))\\
	& \cup \Omega_{k} (k),
\end{align*}
so that 
\begin{align*}
E(P) = & \hspace{0.4cm} \{ \{x, x+1\}, \mid x = -1, u \} \\
	 & \cup \{ \{x, x+2\} \mid x = -1,v-2 \}  \\
	 & \cup \{ \{x, x+d\} \mid x = u-\lfloor d \slash 2 \rfloor, v-\lfloor  d \slash 2 \rfloor -1 \textrm{ where } d = 3, \dots, k-2 \} \\
	 & \cup \{ \{x, x+k-1\} \mid x = u, u+1\} \\
	 & \cup \{ \{x, x+k\}, \mid x = u-1, k \}. 
\end{align*}

Figure \ref{fig_k4andk8} shows this construction for $k = 4$ and $k=8$.
\begin{figure}[h!]
\begin{center}
\includegraphics[scale=1.4]{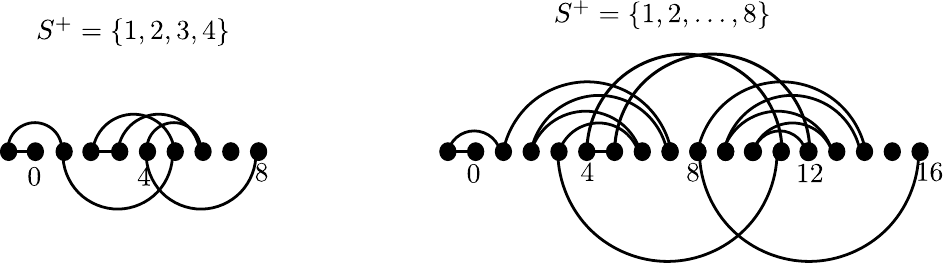}
\caption{Construction of $P$ from Theorem \ref{ThmConsec} when $k=4$ and $k = 8$.}
\label{fig_k4andk8}
\end{center}
\end{figure}

It is straightforward to check that $H_1 = \bigcup_{i \in \mathbb{Z}} (P + 2ki)$ is a Hamilton path in $G$. 
Observe that, for each $d \in \{1,2,\dots,k\}$, the path $P$ has two edges of the form $\{x,x+d\}$, one where $x$ is even and one where $x$ is odd (since $u$ and $v$ are even).  Now it can be checked that $\{H_1, H_1+2, H_1+4, \dots, H_1 + 2k-2 \}$ is a Hamilton decomposition of $G$.
\end{proof}

Although $\Cay(\mathbb{Z}, \pm\{1,2,3,\dots,k\})$ is not admissible when $k \equiv 2,3 \pmod{4}$, the graph $\Cay(\mathbb{Z}, \pm \{1,2,3,\dots,k-1,k+1\})$ is admissible, and we next show that it is also Hamilton-decomposable.

\begin{theorem} \label{ThmConsec_skipk}  If $k$ is a positive integer, then $\Cay(\mathbb{Z}, \pm \{1,2,3,\dots,k-1,k+1\})$ is Hamilton-decomposable if and only if it is admissible.
\end{theorem}

\begin{proof}  Let $G = \Cay(\mathbb{Z}, \pm\{1,2,\dots,k-1,k+1\})$ and observe that $G$ is admissible if and only if $k \equiv 2,3 \smod{4}$.  If $k=2$, then the result follows by Theorem \ref{Thm_ab}.  If $k=3$, define $P =[0,1,-1,3]$ (see Figure \ref{fig_124}) and let $H_1 =  \bigcup_{i\in \mathbb{Z}}\left( P+3i \right)$ so that $ \{H_1, H_1+1, H_1+2\}$ is a Hamilton decomposition of $G$. Thus, we now assume $k \geq 6$.
\begin{figure}[h!]
\begin{center}
\includegraphics[scale=1.4]{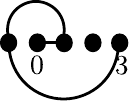}
\caption{Construction of $P$ from Theorem \ref{ThmConsec_skipk} when $k=3$.}
\label{fig_124}
\end{center}
\end{figure}

\textbf{Case 1} Suppose $k \equiv 2 \smod{4}$ and $k \geq 6$.  If $k = 6$, then with $P = [0,2,-3,-2,-5,-1,6]$, $H_1 = \cup_{i \in \mathbb{Z}} (P + 6i)$ is a Hamilton path in $G$ and $\{H_1, H_1+1, H_1+2, H_1 + 3, H_1+4, H_1+5 \}$ is a Hamilton decomposition of $G$.  Now suppose $k \geq 10$ and let $u = \frac{k}{2}$. Define a path with $k$ edges as follows:
\begin{align*}
P := & \hspace{0.4cm} \Omega_0( u-1, -(k-1))\\
 & \cup  \Omega_{-u} ( -2, 3, -4, 5, \dots, -(u-3), u-2)  \\
 & \cup \Omega_{-(u+3)/2} (1) \\
 & \cup \Omega_{-(u+1)/2} ( -u, u+1, -(u+2),u+3 \dots, -(k-3), k-2)\\  
 & \cup \Omega_{-1}(k+1).
\end{align*}
Figure \ref{fig_k6andk10} shows this construction for $k =6$ and $k = 10$.

\begin{figure}[h!]
\begin{center}
\includegraphics[scale=1.4]{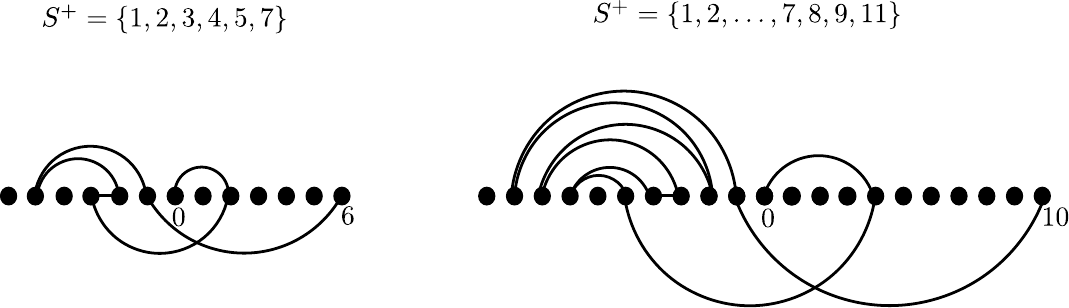}
\caption{Construction of $P$ from Theorem \ref{ThmConsec_skipk} when $k=6$ and $k = 10$.}
\label{fig_k6andk10}
\end{center}
\end{figure}

It is straightforward to check that $H_1 = \bigcup_{i \in \mathbb{Z}} (P+ki)$ is a Hamilton path in $G$ and $\{H_1, H_1+1, H_1+2, \dots, H_1+k-1\}$ is a Hamilton decomposition of $G$.

\textbf{Case 2} Suppose $k \equiv 3 \smod{4}$ and $k \geq 7$.
Let $P$ be the path with $k$ edges defined as follows:
\begin{align*} 
P := & \hspace{4mm} \Omega_0( 1)\\
 & \cup \Omega_1(k-3,5,k-7,9,k-11,13,\dots, 4, k-2)\\
  & \cup \Omega_v (-(k-1))\\
  & \cup \Omega_{v-k+1}(-2, -(k-4), -6, -(k-8), \dots, -(k-5), -3)\\
  & \cup \Omega_{-1}(k+1),
\end{align*}
where $v = (1 + 5 + 9 + \dots + k-2) + (4+8 +12 +\dots+ k-3) = \frac{k+1}{4}(k-2)$.  Figure \ref{fig_k7andk11} shows this construction for $k =7$ and $k = 11$.
\begin{figure}[h!]
\begin{center}
\includegraphics[scale=1.2]{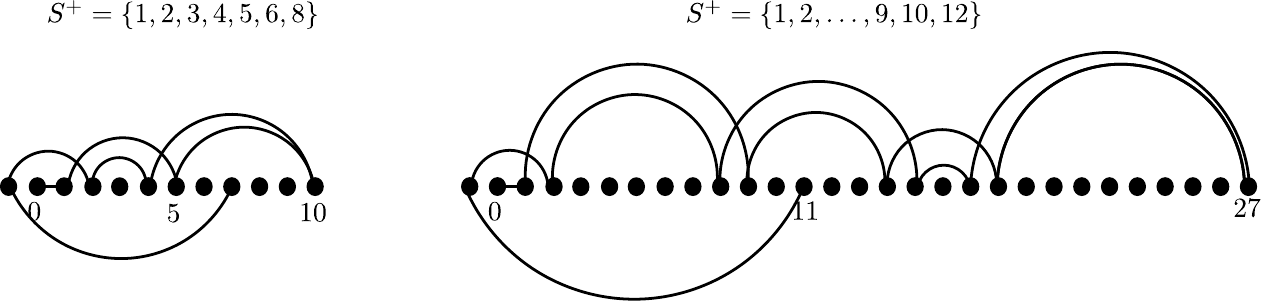}
\caption{Construction of $P$ from Theorem \ref{ThmConsec_skipk} when $k=7$ and $k = 11$.}
\label{fig_k7andk11}
\end{center}
\end{figure}

Note that, modulo $k$, the vertices of $P$ are congruent to:
\begin{align*}
 &0,1,k-2,3, k-4, 5, k-6, 7, \dots, \frac{k+3}{2}, \frac{k-1}{2},\\ 
 & \frac{k+1}{2}, \frac{k-3}{2}, \frac{k+5}{2}, \frac{k-7}{2}, \dots, k-3, 2, k-1, 0.
\end{align*}

Thus, $P$ is a path with endpoints $0$ and $k$ having exactly one vertex from each congruence class modulo $k$ (except that the endpoints are both $0 \smod{k}$).  Hence $H_1 = \bigcup_{i \in \mathbb{Z}} (P+ki)$ is a Hamilton path in $G$.  For each $d \in \{1,2,\dots, k-1,k+1\}$ the path $P$ uses exactly one edge of the form $\{x,x+d\}$.  Thus $\{H_1, H_1+1, H_1+2, \dots, H_1+k-1\}$ is a Hamilton decomposition of $G$.
\end{proof}

Next, we consider graphs $\Cay(\mathbb{Z}, S)$ where $S^+$ consists of consecutive even integers together with $1$. 

\begin{theorem} \label{Thm_consec_evens}
If $t$ is a positive integer, then $\Cay(\mathbb{Z}, \pm \{1, 2,4,6,8,\dots, 2t \})$ is Hamilton-decomposable if and only if it is admissible.
\end{theorem}

\begin{proof} 
Suppose $G = \Cay(\mathbb{Z}, \pm \{1, 2,4,6,8,\dots, 2t \})$ is admissible and let $k = |S^+| = t+1$.  Since $1+2+4+\dots+2t$ is odd, it follows that $k$ is odd and hence $t$ is even.  If $t = 2$, then the result follows by Theorem \ref{ThmConsec_skipk}.  Thus, we can assume $t \geq 4$.


Observe that the elements of $S^+$, namely $1,2,4,\dots,t-2, t, t+2, t+4, \dots 2t$, are congruent to $1,2,4,\dots,t-2, t,-t,-(t-2), \dots, -4, -2$ modulo $k$, respectively.  We define $P$ with $k$ edges as follows, depending on $k \smod{4}$:

If $k \equiv 1 \smod{4}$ then
\begin{align*}
 P := & \hspace{4mm} \Omega_0(1)\\
       & \cup \Omega_{1}(t-2, -(t-4), t-6, -(t-8), \dots, 6, -4) \\
       & \cup \Omega_{(t-2)/2}( 2, t)\\
       & \cup \Omega_{(3t+2)/2}(-2t, 2t-2, -(2t-4),2t-6, \dots, -(t+4), t+2);
\end{align*}
and if $k \equiv 3 \smod{4}$ then
\begin{align*}
P := & \hspace{4mm} \Omega_0(1)\\
     & \cup \Omega_{1} (t-2, -(t-4), t-6, -(t-8), \dots, 4, -2)\\
     & \cup \Omega_{t/2} (-t,  2t)\\
     & \cup \Omega_{3t/2} ( -(2t-2), 2t-4, -(2t-6), 2t-8, \dots, -(t+4), t+2).
\end{align*}

For example, 
\[
P = \left\{ \begin{array}{ll}
  [0, 1, 3, 7, -1, 5] & \textrm{ if } k = 5 \\
 \textrm{$[0, 1, 5, 3, -3, 9, -1, 7]$} & \textrm{ if } k = 7 \\
 \textrm{$[0, 1, 7, 3, 5, 13, -3, 11, -1, 9] $} & \textrm{ if } k = 9 \\
 \textrm{$[0, 1, 9, 3, 7, 5, -5, 15, -3, 13, -1, 11]$} & \textrm{ if } k = 11. \\
\end{array} \right.
\]

Let $u = \frac{k-1}{2}$ and note that, modulo $k$, the vertices of $P$ are congruent to 
\[
\begin{array}{ll}
0,1,k-2,3,\dots, u-1, u+1, u, u+2, u-2, \dots, 2, k-1, 0 & \textrm{ if } k \equiv 1 \ssmod{4}\\
0,1, k-2,3,\dots,u+2, u, u+1, u-1, u+3, \dots, 2, k-1, 0 & \textrm{ if } k \equiv 3 \ssmod{4}.\\
\end{array}
 \]

In either case, $P$ is a path with endpoints $0$ and $k$ having exactly one vertex from each congruence class modulo $k$ (except that the endpoints are both $0 \smod {k}$).  Thus $H_1 = \bigcup_{i \in \mathbb{Z}} (P+ki)$ is a Hamilton path in $G$.  It is straightforward to check that $\{H_1,H_1+1,H_1+2 \dots, H_1 + k-1\}$ is a Hamilton decomposition of $G$.
\end{proof}

We conclude this section on $2k$-valent infinite circulant graphs with a discussion of the case $k=3$.  By Theorem \ref{Cor_k_small_odd_prime}, if $a$ and $b$ are distinct positive integers, not divisible by $3$, then  $G = \Cay(\mathbb{Z}, \pm \{3,a,b\})$ is Hamilton-decomposable if and only if $G$ is admissible.  Note that graphs of the form $\Cay(\mathbb{Z}, \pm\{a,3t,3\})$, where $a \not\equiv 0 \smod{3}$ are admissible but not covered by Theorem \ref{Cor_k_small_odd_prime}, however we have verified by computer that several small admissible cases of this form are Hamilton-decomposable. 

A straightforward corollary of Theorem \ref{Thm_ab} is that if $a, b \in \mathbb{Z}^+$ are odd and relatively prime then $\Cay(\mathbb{Z}, \pm \{1,a,b\})$ is Hamilton-decomposable.  It remains an open problem to determine whether $\Cay(\mathbb{Z}, \pm \{1,a,b\})$ is Hamilton-decomposable when $a,b$ are both even and when $a,b$ are both odd but not relatively prime.  The next result answers this question when $a=2$.

\begin{theorem}\label{Thm_12c}
If $c\geq 3$ is an integer, then $\Cay(\mathbb{Z}, \pm\{1,2,c\})$ is Hamilton-decomposable if and only if it is admissible.
\end{theorem}

\begin{proof} Suppose $G= \Cay(\mathbb{Z}, \pm\{1,2,c\})$ is admissible.  Then $c$ is even and we may assume that $c = 2t$ for some integer $t \geq 2$.  If $t = 2$, then $G$ is Hamilton-decomposable by Theorem \ref{ThmConsec_skipk}.
Thus, we can assume that $t \geq 3$.  

We divide the proof into three cases, $t$ odd, $t \equiv 0 \smod{4}$ and $t \equiv 2 \smod{4}$ and provide a construction for a Hamilton decomposition in each case.  The constructions are based on a starter path $P$ which has $3t$ edges, using $t$ edges of each length $1$, $2$ and $2t$.  
In each case it can be checked that $H_1 =  \bigcup_{i\in \mathbb{Z}}\left( P+3ti \right)$ is a Hamilton path and that $\{H_1, H_1+t, H_1+2t\}$ is a Hamilton decomposition of $G$. \\

In the cases below, we use the following notation:
\begin{align*}
A_v &:= \Omega_v(1,2t,1,-2t)  \\
B_v &:= \Omega_v(2t, -2, -2t, -2) \\
C_v &:= \Omega_v(2t, 2, -2t, 2). 
\end{align*}

\noindent \textbf{Case 1:} $t$ odd\\
Define
\begin{align*}
P =  \bigcup_{i=0}^{(t-3)\slash 2} A_{2i} \hspace{0.2cm} \cup
 [t-1,t+1,t+3,\dots,2t-2,2t, 2t-1,2t-3,2t-5, \dots, t+2, t, 3t].
\end{align*}
Thus, $P$ is of the form
\begin{align*}
\bigcup_{i=0}^{(t-3)\slash 2} A_{2i} \hspace{0.2cm} \cup \Omega_{t-1}(2,2,\dots,2) \cup \Omega_{2t}(-1, -2, -2, -2, \dots, -2) \cup \Omega_t(2t),
\end{align*}
and
\begin{align*}
E(P) = & \{ \{x,x+1\} \mid x \in \{0,2,\dots,t-3, 2t-1,2t+1,2t+3,\dots, 3t-2\} \} \cup \\
		& \{ \{x,x+2\} \mid x \in \{t-1,t, t+1,\dots,2t-2\} \} \cup \\
		& \{ \{x, x+2t\} \mid x \in \{1,2,\dots,t\} \}.
\end{align*}		
Figure \ref{fig_t3_t5} shows this construction for $t=3$ and $t=5$.

\begin{figure}[h!]
\begin{center}
\includegraphics[scale=1.3]{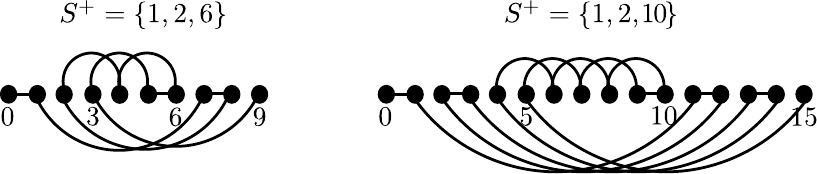}
\caption{Construction of $P$ from Theorem \ref{Thm_12c} when $t=3$ and $t = 5$.}
\label{fig_t3_t5}
\end{center}
\end{figure}

\noindent \textbf{Case 2: $t \equiv 0 \smod{4}$}\\

\noindent If $t = 4$, define $P =[0,1,9,11,3,5,6,7,8,10,2,4,12]$.  For $t \geq 8$, we define  
\begin{align*}
P = [& 0,1, 2t+1, 2t+3, 3,2, 2t+2,2t, 2t-1,2t-2,\dots,t+5, t+3, t+4,t+2,t+1,t-1]\\
	&  \bigcup_{i=0}^{(t-12)\slash 4} B_{t-1-4i} \hspace{0.2cm} \cup [7, 2t+7, 2t+5, 5,4] \hspace{0.2cm} \bigcup_{i=1}^{(t-4)\slash 4} C_{4i} \hspace{0.2cm}  \cup [t,3t].
\end{align*}
Thus, $P$ is of the form 
\begin{align*}
& \hspace{2mm} \Omega_0(1,2t, 2, -2t, -1, 2t, -2) \hspace{2mm} \cup \hspace{2mm} \Omega_{2t}(-1,-1, \dots, -1) \hspace{2mm} \cup \hspace{2mm} \Omega_{t+5}(-2,1,-2,-1,-2) \\
	&  \bigcup_{i=0}^{(t-12)\slash 4} B_{t-1-4i} \hspace{2mm} \cup \hspace{2mm} \Omega_7(2t, -2, -2t, -1) \hspace{0.2cm} \bigcup_{i=1}^{(t-4)\slash 4} C_{4i} \hspace{2mm}  \cup \hspace{2mm} \Omega_t(2t),
\end{align*}
and
\begin{align*}
E(P) = & \hspace{0.4cm}  \{ \{x,x+1\} \mid x \in \{0,2,4,t+1,t+3\} \textrm{ or } t+5 \leq x \leq 2t-1 \}  \\
		& \cup \{ \{x,x+2\} \mid 6 \leq x \leq t+3 \textrm{ where } x \equiv 2,3 \ssmod{4} \textrm{ or } \\
		& \hspace{2.8cm}   2t \leq x \leq  3t-3 \textrm{ where } x \equiv 0,1 \ssmod{4}  \}  \\
		& \cup \{ \{x, x+2t\} \mid 1 \leq x \leq t  \}.
\end{align*}	
Figure \ref{fig_t4_t8} shows the construction for $t = 4$ and $t=8$.\\
\begin{figure}[h!]
\begin{center}
\includegraphics[scale=1.3]{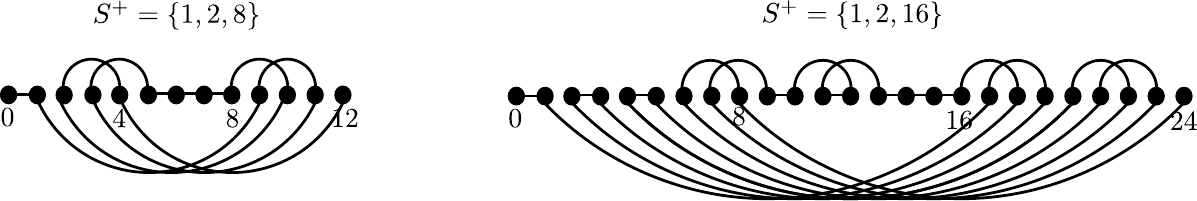}
\caption{Construction of $P$ from Theorem \ref{Thm_12c} when $t=4$ and $t =8$.}
\label{fig_t4_t8}
\end{center}
\end{figure}

\noindent \textbf{Case 3: $t \equiv 2 \smod{4}$}\\

\noindent The case $t =2$ is done, so here $t \geq 6$. We define
\begin{align*}
P = [& 0,1, 2t+1, 2t+3, 3,2, 2t+2,2t, 2t-1,2t-2,\dots,t+5, t+3, t+4,t+2,t+1,t-1]\\
	&  \bigcup_{i=0}^{(t-10)\slash 4} B_{t-1-4i} \hspace{0.2cm} \cup [5, 2t+5, 2t+4, 4,6] \hspace{0.2cm} \bigcup_{i=1}^{(t-6)\slash 4} C_{4i+2} \hspace{0.2cm}  \cup [t,3t].
\end{align*}
Thus, $P$ is of the form 
\begin{align*}
& \hspace{2mm} \Omega_0(1,2t,2,-2t,-1,2t,-2) \hspace{2mm} \cup \hspace{2mm} \Omega_{2t}(-1,-1, \dots, -1) \hspace{2mm} \cup \hspace{2mm} \Omega_{t+5}(-2,1,-2,-1,-2) \\
	&  \bigcup_{i=0}^{(t-10)\slash 4} B_{t-1-4i} \hspace{2mm} \cup \hspace{2mm} \Omega_5(2t,-1,-2t,2) \hspace{2mm} \bigcup_{i=1}^{(t-6)\slash 4} C_{4i+2} \hspace{2mm}  \cup \hspace{2mm} \Omega_t(2t),
\end{align*}
and
\begin{align*}
E(P) = & \hspace{0.4cm}  \{ \{x,x+1\} \mid x \in \{0,2,t+1,t+3,2t+4\}  \textrm{ or } t+5 \leq x \leq 2t-1 \}  \\
		& \cup \{ \{x,x+2\} \mid x \in \{2t, 2t+1\} \textrm{ or } \\
		& \hspace{2.8cm}  4 \leq x \leq t+3 \textrm{ where } x \equiv 0,1 \ssmod{4} \textrm{ or }\\
		& \hspace{2.8cm}   2t+6 \leq x \leq  3t-3 \textrm{ where } x \equiv 2,3 \ssmod{4}  \}  \\
		& \cup \{ \{x, x+2t\} \mid 1 \leq x \leq t  \}.
\end{align*}	
Figure \ref{fig_t6_t10} shows this construction for $t=6$ and $t = 10$.
\begin{figure}[h!]
\begin{center}
\includegraphics[scale=1.1]{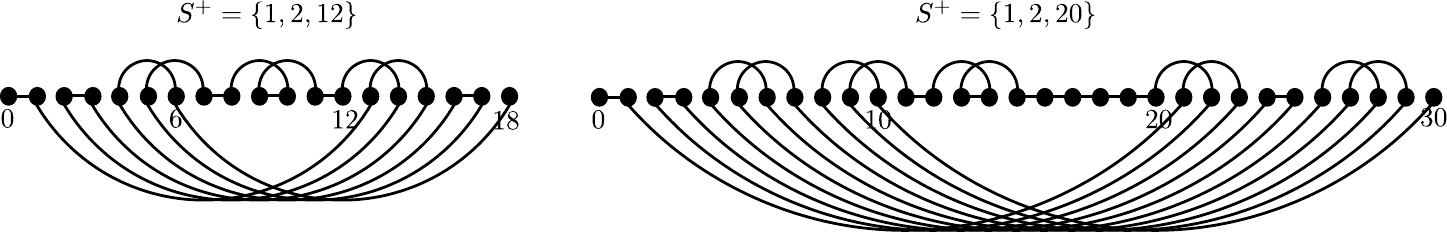}
\caption{Construction of $P$ from Theorem \ref{Thm_12c} when $t=6$ and $t=10$.}
\label{fig_t6_t10}
\end{center}
\end{figure}

\end{proof}

\noindent{\bf Acknowledgement:} The authors acknowledge the support of the Australian Research Council (grants DP150100530, DP150100506, DP120100790 and DP130102987) and the EPSRC (grant EP/M016242/1).  Some of this research was undertaken while Webb was an Ethel Raybould Visiting Fellow at The University of Queensland.



\begin{thebibliography}{10}

\bibitem {A} B.~Alspach, Research problem 59, \textit{Discrete Math.}, \textbf{50} (1984), 115.

\bibitem{AlsBryKre}
B.~Alspach, D.~Bryant and D.L.~Kreher, 
Vertex-transitive graphs of prime-squared order are Hamilton-decomposable, 
{\it J. Combin. Des.}, 
{\bf 22} (2014), 12--25.

\bibitem {BeFaMa} J.-C.~Bermond, O.~Favaron and M.~Maheo, 
Hamiltonian decomposition of Cayley graphs of degree 4, 
\textit{J. Combin. Theory Ser. B}, 
\textbf{46} (1989), 142--153.

\bibitem{BryDea} D.~Bryant and M.~Dean, 
Vertex-transitive graphs that have no Hamilton decomposition, 
{\it J. Combin. Theory Ser. B}, 
{\bf 114} (2015), 237--246.

\bibitem {BrMa} D.~Bryant and G.~Martin,
 Some results on decompositions of low degree circulant graphs,
 \textit{Australas. J. Combin.},
 \textbf{45} (2009), 251--261.

\bibitem {CF} S.~Capparelli and A.~Del Fra,
 Hamiltonian paths in the complete graph with edge-lengths 1,2,3,
  {\em Electron. J. Combin.}, 
  \textbf{17} (2010), \#44.

\bibitem{CheQui}
C.C.~Chen and N.F.~Quimpo, 
On strongly Hamiltonian abelian group graphs. Combinatorial mathematics, VIII (Geelong, 1980), pp. 23--34, Lecture Notes in Math., 884, Springer, Berlin-New York, 1981. 

\bibitem{Dea1}
M.~Dean,
On Hamilton cycle decomposition of $6$-regular circulant graphs,
{\it Graphs Combin.},
{\bf 22} (2006), 331--340. 

\bibitem{Dea2}
M.~Dean, 
Hamilton cycle decomposition of $6$-regular circulants of odd order, 
{\it J. Combin. Des.},
{\bf 15} (2007), 91--97.

\bibitem {DJ} J.H.~Dinitz and S.R.~Janiszewski,
 On Hamiltonian Paths with Prescribed edge lengths in the Complete Graph,
  \textit{Bull. Inst. Combin. Appl}, {\bf 57} (2009), 42--52.


\bibitem {DuJuWi} D.~Dunham, D.S.~Jungreis and D.~Witte,
 Infinite Hamiltonian paths in Cayley digraphs of hyperbolic symmetry groups,
  \textit{Discrete Math.},
   \textbf{143} (1995), 1--30.

\bibitem{FanLicLiu}
C.~Fan, D.R.~Lick and J.~Liu, 
Pseudo-Cartesian product and Hamiltonian decompositions of Cayley graphs on abelian groups,
{\it Discrete Math.},
{\bf 158} (1996), 49--62. 

\bibitem {HeMa} S.~Herke and B.~Maenhaut,
 Perfect $1$-Factorisations of Circulants with Small Degree, 
  {\em Electron. J. Combin.},
   \textbf{20(1)} (2013), \#P58.

\bibitem {HR} P.~Horak and A.~Rosa, 
On a problem of Marco Buratti,
 {\em Electron. J. Combin.},
  \textbf{16} (2009), \#R20.

\bibitem {JunD}  D.~Jungreis, 
Hamiltonian paths in Cayley digraphs of finitely-generated infinite abelian groups, \textit{Discrete Math.}, 
\textbf{78} (1989), no. 1--2, 95--104. 

\bibitem {JunI} I.L.~Jungreis, 
Infinite Hamiltonian paths in Cayley digraphs,
 \textit{Discrete Math.},
  \textbf{54} (1985), no. 2, 167--180. 

\bibitem{Liu1} 
J.~Liu, 
Hamiltonian decompositions of Cayley graphs on abelian groups, 
{\it Discrete Math.}, 
{\bf 131} (1994), 163--171.

\bibitem{Liu2}
J.~Liu, 
Hamiltonian decompositions of Cayley graphs on abelian groups of odd order, 
{\it J. Combin. Theory Ser. B},
{\bf 66} (1996), 75--86.

\bibitem{Liu3}
J.~Liu, 
Hamiltonian decompositions of Cayley graphs on abelian groups of even order, 
{\it J. Combin. Theory Ser. B},
{\bf 88} (2003), 305--321.

\bibitem {Lo} L.~Lov\'asz,  
The factorization of graphs,
  In \textit{ Combinatorial Structures and their Applications (Proc. Calgary Internat. Conf., Calgary, Alta., 1969)},
   (1970), Gordon and Breach, New York, 243--246.

\bibitem {NW} C.St.J.A.~Nash-Williams, 
Abelian groups, graphs, and generalized knights, 
\textit{Proc. Camb. Phil. Soc.}, 
\textbf{55} (1959), 232--238.

\bibitem {W} E.E.~Westlund,
 Hamilton decompositions of 6-regular Cayley graphs on even Abelian groups with involution-free connections sets,
  \textit{Discrete Math.}, 
  \textbf{331} (2014), 117--132. 

\bibitem{Wes2} 
E.E.~Westlund, 
Hamilton decompositions of certain $6$-regular Cayley graphs on Abelian groups with a cyclic subgroup of index two, 
{\it Discrete Math.},
{\bf 312} (2012), 3228--3235. 

\bibitem{WesLiuKre}
E.E.~Westlund, J.~Liu and D.L.~Kreher, 
$6$-regular Cayley graphs on abelian groups of odd order are Hamiltonian decomposable,
{\it Discrete Math.},
{\bf 309} (2009), 5106--5110.

\bibitem {Wi} D.~Witte, 
Hamilton-decomposable graphs and digraphs of infinite valence, 
\textit{Discrete Math.},
 \textbf{84} (1990), 87--100.

\bibitem {WiGa} D.~Witte and J.A.~Gallian, 
A survey: Hamiltonian cycles in Cayley graphs,
 \textit{Discrete Math.},
  \textbf{51} (1984), 293--304.

\bibitem {ZH} F.~Zhang and Q.~Huang, 
Infinite circulant graphs and their properties,
 \textit{Acta Math. Appl. Sinica (English Ser.)},
  \textbf{11} (1995), no.3, 280--284.

\end{thebibliography}
\end{document}